\documentclass[a4paper,11pt]{amsart}
\usepackage{amsfonts,amssymb,amsthm,cite,amsmath,amstext}
\usepackage{color}
\usepackage[colorlinks,linkcolor=black,citecolor=black]{hyperref}
\usepackage{comment}
\usepackage{framed}
\usepackage{calligra} % for dedication

\definecolor{shadecolor}{gray}{0.875}

\newtheorem{thrm}{Theorem}[section]
\newtheorem{thrmx}{Theorem}

\newtheorem{corx}{Corollary}

\newtheorem{lem}[thrm]{Lemma}

\newtheorem{conj}[thrm]{Conjecture}
\newtheorem*{griffconj}{Griffiths Conjecture}

\theoremstyle{definition}
\newtheorem{defn}[thrm]{Definition}

\newtheorem{exmple}[thrm]{Example}
\newtheorem{rmk}[thrm]{Remark}
\newtheorem{ques}[thrm]{Question}

\DeclareMathOperator{\vol}{vol}

\DeclareMathOperator{\Id}{Id}

\DeclareMathOperator{\End}{End}
\DeclareMathOperator{\trace}{trace}

\title{On the positivity of high-degree Schur classes of an ample vector bundle}
\author{Jian Xiao}
\date{}

\begin{document}
\maketitle

%\begin{dedication}
%In memory of ...
%\end{dedication}

\begin{abstract}
Let $X$ be a smooth projective variety of dimension $n$, and let $E$ be an ample vector bundle over $X$. We show that any non-zero Schur class of $E$, lying in the cohomology group of bidegree $(n-1, n-1)$, has a representative which is strictly positive in the sense of smooth forms. This conforms the prediction of Griffiths conjecture on the positive polynomials of Chern classes/forms of an ample vector bundle on the form level, and thus strengthens the celebrated positivity results of Fulton-Lazarsfeld for certain degrees.
\end{abstract}

\tableofcontents

\section{Introduction}
\subsection{Positive vector bundles}
Let $X$ be a smooth projective variety of dimension $n$ defined over complex numbers.
Let $E$ be a holomorphic vector bundle of rank $r$ over $X$, endowed with a smooth Hermitian metric $h$. If $D_{E, h}$ is the Chern connection of $(E, h)$, and $\Theta_{E, h} = D^2 _{E, h} \in \Lambda^{1,1} (X, \End(E))$ its curvature, then the corresponding Chern forms $c_k (E, h)$ are computed formally as follows:
\begin{equation*}
  \det(\Id + \frac{it}{2\pi} \Theta_{E, h}) = \sum_{k=0}^r c_k (E, h) t^k,
\end{equation*}
or equivalently,
\begin{equation*}
c_k (E, h) = \trace (\wedge ^k \frac{i}{2\pi} \Theta_{E, h}).
\end{equation*}
The Chern form $c_k (E, h)$ is a globally defined real $d$-closed $(k, k)$ form on $X$. By the Chern-Weil theory, the $k$-th Chern class of $E$, denoted by $c_k (E)$, can be represented by the smooth form $c_k (E, h)$ and lies in $H^{k, k}(X, \mathbb{Z})$.

Let $\pi: \mathbb{P}(E) \rightarrow X$ be the projective bundle of hyperplanes of $E$, and $\mathcal{O}_E (1)$ the tautological line bundle on $\mathbb{P}(E)$:
\begin{equation*}
0\rightarrow S \rightarrow \pi^* E \rightarrow \mathcal{O}_E (1) \rightarrow 0.
\end{equation*}
Then $E$ is called ample if the line bundle $\mathcal{O}_E (1)$ is ample on $\mathbb{P}(E)$.
Corresponding to the ampleness of $E$, there is a closely related differential-geometric positivity notion, which we describe below.
Denote by $(e_1, ..., e_r)$ a local normal frame of $E$ over a coordinate patch $\Omega \subset X$
with a local holomorphic coordinates system $(z_1, ..., z_n)$. Then over $\Omega$ the curvature tensor $\Theta_{E, h}$ can be written as
\begin{equation*}
  \Theta_{E, h} = \sum_{\lambda, \mu =1} ^r \sum_{j, k =1} ^n c_{jk\lambda \mu} dz_j \wedge d\bar{z}_k \otimes e_\lambda ^* \otimes e_\mu,\ \text{satisfying}\ c_{jk\lambda\mu}=\overline{c_{kj\mu \lambda}}.
\end{equation*}
To $\Theta_{E, h}$ correspond a natural Hermitian form $\theta_{E, h}$ on $TX \otimes E$ defined by
\begin{equation*}
  \theta_{E, h} = \sum_{j,k,\lambda, \mu}  c_{jk\lambda \mu} (dz_j \otimes e_\lambda ^*) \otimes (\overline{dz_k \otimes e_\mu ^*}),
\end{equation*}
such that at a point $x\in \Omega$ we have
\begin{equation*}
  \theta_{E, h} (u, u) = \sum_{j,k,\lambda, \mu}  c_{jk\lambda \mu} u_{j \lambda} \overline{u_{k \mu}}, \ u \in T_x X \otimes E_x.
\end{equation*}
The vector bundle $E$ endowed with the metric $h$ is called Griffiths-positive, if for any $\xi \in T_x X, \xi\neq 0$ and $s\in E_x, s\neq 0$,
\begin{equation*}
  \theta_{E, h} (\xi \otimes s, \xi \otimes s) >0.
\end{equation*}

It is a celebrated and still widely open problem of Griffiths \cite{griffithsConj} that the above algebraic-geometric ampleness and differential-geometric Griffiths-positivity are equivalent. One direction is clear. Endowing $\mathcal{O}_E (1)$ with the induced metric from $h$, it is easy to see that:
\begin{equation*}
E\ \text{being Griffiths-positive}\ \Rightarrow\ E\ \text{being ample}.
\end{equation*}

\subsection{Griffiths conjecture on numerically positive polynomials}
For an ample vector bundle $E$, its first Chern class $c_{1} (E)=c_1 (\det E)$ has a representative which is a K\"ahler metric, i.e., a smooth strictly positive $(1,1)$ form. It is natural to ask whether this also holds for the $k$-th Chern classes $c_k (E)$. We will confirm this for $c_{n-1} (E)$ whenever it is non-zero, i.e., when $E$ is an ample vector bundle of rank $r\geq n-1$.

Indeed, in the seminal paper \cite{griffithsConj}, Griffiths asked the following question, which we formulate as follows (see \cite{gulerThesis} for a nice exposition).
\begin{ques}[Griffiths] \label{ques griff}
Let $P\in \mathbb{Q}[c_1, ..., c_r]$ be a homogeneous polynomial of weighted degree $k$, the variable $c_i$ being assigned weight $i$.
\begin{enumerate}
  \item Assume that $(E, h)$ is an arbitrary Griffiths-positive bundle, characterize the polynomial $P$ such that the cohomology class $P(c_1 (E), ..., c_r (E))$ is numerically positive, that is, for any irreducible subvariety $V$ of dimension $k$ in $X$,
      \begin{equation*}
        \int_V P(c_1 (E), ..., c_r (E)) >0.
      \end{equation*}
  \item Assume that $(E, h)$ is an arbitrary Griffiths-positive bundle, characterize the polynomial $P$ such that the $(k, k)$ form $P(c_1 (E, h), ..., c_r (E, h))$ is positive.
\end{enumerate}
\end{ques}

\begin{rmk}
The first part of Griffiths question is on the cohomology level, while the second part is on the form level which is a priori stronger than the first part.
\end{rmk}

It is conjectured explicitly in \cite[Conjecture 0.7]{griffithsConj} that the desired polynomials in Question \ref{ques griff} (1) are given by Griffiths-positive polynomials. For the definition of Griffiths-positive polynomials, see \cite[Definition 5.9]{griffithsConj}. The numerical positivity of Chern classes for ample vector bundles is also conjectured in \cite{HartshorneAmpleBdl}.

Let $\Lambda(k, r)$ be the set of all partitions of $k$ by non-negative integers $\leq r$. %i.e., any element $\lambda \in \Lambda(k, r)$ is a sequence
%\begin{equation*}
%  r\geq \lambda_1 \geq \lambda_2 \geq...\geq \lambda_k \geq 0
%\end{equation*}
%satisfying $|\lambda|=\sum_{i=1} ^k \lambda_i =k$.
Associated to each partition $\lambda \in \Lambda(k, r)$, there is a Schur polynomial $s_\lambda \in \mathbb{Q}[c_1,...,c_r]$ of weighted degree $k$ defined as the determinant:
\begin{equation*}
s_\lambda (c_1,...,c_r)=\det [c_{\lambda_j -j +l}]_{1\leq j, l \leq k}.
\end{equation*}
The space of homogeneous polynomials $P\in \mathbb{Q}[c_1, ..., c_r]$ of weighted degree $k$ is spanned by Schur polynomials, one can write such $P$ uniquely as a linear combination of the $s_\lambda$:
\begin{equation*}
  P=\sum_{\lambda \in \Lambda(k, r)} a_\lambda (P) s_\lambda.
\end{equation*}

On the cohomology level, not only for Griffiths-positive vector bundles but also for ample vector bundles, Question \ref{ques griff} (1) was answered completely in the celebrated work of Fulton-Lazarsfeld \cite{fultonlazarsfeldPositive}. This extends the previous works of Kleiman \cite{kleimanAmplesurface} for surfaces, Bloch-Gieseker \cite{blochgieseker} for Chern classes, Gieseker \cite{giesekerample} for monomials of Chern classes and Usui-Tango \cite{usuitango} for ample and globally generated bundles.

\begin{thrm}[Fulton-Lazarsfeld]
A weighted homogeneous polynomial $P$ is numerically positive for ample vector bundles of rank $r$ if and only if
\begin{equation*}
  P\neq 0\ \text{and}\ a_\lambda (P)\geq 0 \ \text{for all}\ \lambda \in \Lambda(k, r).
\end{equation*}
Equivalently, the Schur polynomials $s_\lambda, \lambda \in \Lambda(k, r)$ span the cone of numerically positive polynomials.
\end{thrm}

Moreover, it is proved that a non-zero weighted homogeneous polynomial $P\in \mathbb{Q}[c_1, ..., c_r]$ is Griffiths-positive if and only if $P$ is a non-trivial non-negative linear combination of Schur polynomials (see \cite[Appendix A]{fultonlazarsfeldPositive}).

It should be noted that Fulton-Lazarsfeld's positivity results are extended to nef vector bundles on compact K\"ahler manifolds in \cite{DPSneftangent}.

Regarding Question \ref{ques griff} (2), Griffiths conjectured \cite[Page 247]{griffithsConj} that such polynomials are also given by Griffiths-positive polynomials. Combining with \cite[Appendix A]{fultonlazarsfeldPositive}, Griffiths conjecture can be formulated as follows:

\begin{griffconj}\label{griff conj}
Let $P\in \mathbb{Q}[c_1,...,c_r]$ be a weighted homogeneous polynomial, then the forms $P(c_1 (E, h), ..., c_r (E, h))$ are positive for any Griffiths-positive vector bundles $(E, h)$ over any smooth projective variety $X$ if and only if $P$ is a  non-negative linear combination of Schur polynomials.
\end{griffconj}

For the above conjecture, Griffiths commented that ``this will require a better understanding of the algebraic properties of the curvature form $\Theta$'', and verified that for a Griffiths-positive vector bundle $(E, h)$ of rank $2$, the second Chern form $c_2 (E, h)$ is positive.
See also \cite{gulersegre, diveriosegre, pingaliChernForm, liping2020nonnegative} for some related results and progress. In the general case, the conjecture is still widely open.

\subsection{The main result}

Motivated by the well known Nakai-Moishezon criterion, we are interested in the relation between the ampleness, numerical positivity and pointwise positivity of Chern classes for a vector bundle.
We propose the following conjecture which is in the intermediate of Fulton-Lazarsfeld's positivity theorem and Griffiths conjecture.

\begin{conj}\label{conj pos}
Let $X$ be a smooth projective variety of dimension $n$, and
$E$ an ample (or Griffiths-positive) vector bundle of rank $r$ over $X$. Let $s_\lambda (c_1,...,c_r)$ be a Schur polynomial, then the class $s_\lambda (c_1 (E), ..., c_r (E))$ has a smooth positive representative.
\end{conj}

By \cite{fultonNakaicriterion}, even on $\mathbb{P}^2$ the existence of positive representatives does not imply the ampleness of $E$.

\begin{rmk}
It is unclear whether for every smooth representative $\Phi$ of $s_\lambda (c_1 (E), ..., c_r (E))$ there is a smooth Hermitian metric $h$ on $E$ such that $\Phi= s_\lambda (c_1 (E,h), ..., c_r (E,h))$. Nevertheless, see \cite{pingaliChernForm} for some results on a surface.
\end{rmk}

In this note, we confirm Conjecture \ref{conj pos} when the weighted degree $|\lambda|=n-1$ by showing that the Schur class $s_\lambda$ has very strong pointwise positivity.

\begin{thrmx}\label{intro main thrm}
Let $X$ be a smooth projective variety of dimension $n$, and let $E$ be an ample vector bundle of rank $r$ on $X$. Then for any Schur polynomial $s_\lambda (c_1,...,c_r)$ with $|\lambda|= n-1$, whenever it is non-zero the cohomology class $s_\lambda(c_1 (E), ..., c_r (E))$ has a smooth representative which is strictly positive in the sense of $(n-1, n-1)$ forms.
\end{thrmx}

A $d$-closed strictly positive $(n-1, n-1)$ form on $X$ is also called a balanced metric in differential-geometric category.

\begin{exmple}
Let $\lambda \in \Lambda(k, r)$ be a partition of $k$. We list some simple Schur polynomials. By definition, we have:
\begin{equation*}
s_\lambda(c_1 (E), ..., c_r (E))=
\left\{
  \begin{array}{ll}
    c_{k} (E), &\ \text{when}\ \lambda = (k, 0,...,0); \\
    \text{Segre class}\ s_k (E),  &\ \text{when}\ \lambda = (1, 1,...,1,0,...,0);\\
    c_j(E) \cdot c_{k-j} (E) -c_{j-1} (E)\cdot c_{k+1-j} (E),  &\ \text{when}\ \lambda = (k-j, j,0,...,0).\\
  \end{array}
\right.
\end{equation*}
\end{exmple}

Since a numerically positive polynomial is a non-negative linear combination of Schur polynomials, Theorem \ref{intro main thrm} strengthens the positivity results of \cite{fultonlazarsfeldPositive} for $(n-1, n-1)$ classes pointwisely.

\begin{exmple}
It is proved in \cite{fultonlazarsfeldPositive} that the product of numerically positive polynomials is again numerically positive. In particular, any monomial $c_I = c_1 ^{i_1}\cdot c_2 ^{i_2}\cdots c_r ^{i_r}$ and
\begin{equation*}
  c_1 ^k - c_k = \sum_{j=1} ^k c_1 ^{k-j} (c_1 c_{j-1} - c_j)
\end{equation*}
are numerically positive for ample vector bundles\footnote{The classes $c_I$ and $c_1 ^k - c_k$ (and its variant) are studied in \cite{giesekerample} and \cite{DPSneftangent, liping2020nonnegative} respectively.}.
Applying Theorem \ref{intro main thrm} implies that whenever non-zero the classes $c_I (E)$ with $\sum_{j=1} ^r j i_j = n-1$, $c_1 ^{n-1} (E) - c_{n-1} (E)$ have smooth representative which are strictly positive.
\end{exmple}

\begin{rmk}
It is possible that $s_\lambda(c_1 (E), ..., c_r (E)), |\lambda|=n-1$ has a smooth representative which is strictly positive in the sense of $(n-1, n-1)$ forms even if $E$ is not ample. For example, let $F$ be an ample vector bundle of rank $r\geq n-1$ and let $E = F \oplus \mathcal{O}$, where $\mathcal{O}$ is the trivial line bundle. Then $c_{n-1} (E) = c_{n-1} (F)$, but $E$ is not ample.
\end{rmk}

For general $k$, by restricting $E$ to a subvariety of dimension $k+1$, we get:

\begin{corx}\label{intro cor1}
Let $X$ be a smooth projective variety of dimension $n$, and let $E$ be an ample vector bundle of rank $r$ on $X$. Let $V \subset X$ be an irreducible smooth subvariety of dimension $k+1$, given by the complete intersection of very ample divisors $H_1, ..., H_{n-k-1}$. Then for any Schur polynomial $s_\lambda (c_1,...,c_r)$ with $|\lambda|= k$, whenever it is non-zero the cohomology class $s_\lambda(c_1 (E), ..., c_r (E))\cdot [V]$ has a smooth representative which is strictly positive in the sense of $(n-1, n-1)$ forms.
\end{corx}

\begin{rmk}
Theorem \ref{intro main thrm} and Corollary \ref{intro cor1} also hold for $\mathbb{R}$-twisted ample vector bundles.
\end{rmk}

In Section \ref{sec prel} we present some notions and basic results that are needed in the proof of our main result. In Section \ref{sec proof}, we give the proof of the main result.

\section{Preliminaries} \label{sec prel}

\subsection{$\mathbb{R}$-twisted vector bundles}
In the proof of the main result, we need to apply a perturbation argument by using the notion of $\mathbb{R}$-twisted vector bundles. We give a very brief review on $\mathbb{R}$-twisted vector bundles. The reader can find more facts in \cite[Section 6.2]{lazarsfeldPosII} and the references therein.

\begin{defn}
A $\mathbb{R}$-twisted vector bundle $E\langle \delta \rangle$ on $X$ is an ordered pair consisting of a vector bundle $E$ on $X$, defined up to isomorphism, and a cohomology class $\delta\in H^{1,1} (X, \mathbb{R})$. The rank of $E\langle \delta \rangle$ is the rank of $E$.
\end{defn}

A $\mathbb{R}$-twisted vector bundle $E\langle \delta \rangle$ is understood as a formal object, yet when $\delta=c_1 (L)$ for some line bundle, $E\langle \delta \rangle$ can be considered as $E\otimes L$.

Let $\pi: \mathbb{P}(E) \rightarrow X$ be the projective bundle of hyperplanes of $E$, and $\mathcal{O}_E (1)$ the tautological line bundle:
\begin{equation*}
  0\rightarrow S \rightarrow \pi^* E \rightarrow \mathcal{O}_E (1) \rightarrow 0,
\end{equation*}
where $S_{(x, [l])} = l^{-1} (0) \subset E_x, l\in E_x ^*$.
If $L$ is a holomorphic line bundle on $X$, then $\mathbb{P}(E\otimes L) \cong \mathbb{P}(E)$ by an isomorphism under which $\mathcal{O}_{\mathbb{P}(E\otimes L)} (1)$ corresponds to $\mathcal{O}_{\mathbb{P}(E)} (1) \otimes \pi^* L$ on $\mathbb{P}(E)$. This motivates the following:

\begin{defn}
A $\mathbb{R}$-twisted vector bundle $E\langle \delta \rangle$ on $X$ is called ample (resp. nef) if $\eta_E + \pi^* \delta$ is a K\"ahler (resp. nef) class on $\mathbb{P}(E)$, where $\eta_E = c_1 (\mathcal{O}_{\mathbb{P}(E)} (1))$.
\end{defn}

The natural operations on vector bundles extend to the $\mathbb{R}$-twisted situation, here we just list a few.

\begin{defn}

The tensor of two $\mathbb{R}$-twisted vector bundles is defined by
\begin{equation*}
E_1 \langle \delta_1 \rangle \otimes E_2 \langle \delta_2 \rangle = (E_1 \otimes E_2)\langle \delta_1 + \delta_2 \rangle,
\end{equation*}
in particular, if $E_2$ is a trivial line bundle then $E_1 \langle \delta_1 \rangle \otimes E_2 \langle \delta_2 \rangle = E_1 \langle \delta_1 + \delta_2 \rangle$.

The direct sum of two $\mathbb{R}$-twisted vector bundles with the same twisting class is defined by
\begin{equation*}
E_1 \langle \delta \rangle \oplus E_2 \langle \delta \rangle = (E_1 \oplus E_2)\langle \delta \rangle.
\end{equation*}

A quotient of $\mathbb{R}$-twisted vector bundle $E\langle \delta \rangle$ is a $\mathbb{R}$-twisted vector bundle $Q\langle \delta \rangle$, where $Q$ is a quotient of $E$. Subbundles of $E\langle \delta \rangle$ are defined similarly.

Let $f: Y \rightarrow X$ be a morphism between two projective varieties, then we define the pullback of a $\mathbb{R}$-twisted vector bundle $E\langle \delta \rangle$ on $X$ to be the $\mathbb{R}$-twisted vector bundle $f^* (E\langle \delta \rangle) = (f^* E)\langle f^*\delta \rangle $ on $Y$.
\end{defn}

\begin{rmk}\label{perturb ample}
From the above discussions, if $E\langle \delta \rangle$ is ample, then for $\varepsilon\in H^{1,1}(X, \mathbb{R})$ small enough (by endowing the space $H^{1,1}(X, \mathbb{R})$ with some norm), the bundle $E\langle \delta + \varepsilon \rangle$ is still ample. The pullback of a nef $\mathbb{R}$-twisted vector bundle by a morphism is nef, and the pullback of an ample $\mathbb{R}$-twisted vector bundle by a finite morphism is ample.
\end{rmk}

\begin{rmk}
The positivity results of \cite{fultonlazarsfeldPositive, blochgieseker, DPSneftangent} also hold in the $\mathbb{R}$-twisted setting.
\end{rmk}

\subsection{Schur polynomials and Schur classes} \label{sec schur}

Let $\Lambda(n, r)$ be the set of all partitions of $n$ by non-negative integers $\leq r$, i.e., any element $\lambda \in \Lambda(n, r)$ is a sequence
\begin{equation*}
  r\geq \lambda_1 \geq \lambda_2 \geq...\geq \lambda_n \geq 0
\end{equation*}
satisfying $|\lambda|=\sum_{i=1} ^n \lambda_i =n$. Associated to each partition $\lambda \in \Lambda(n, r)$, there is a Schur polynomial $s_\lambda \in \mathbb{Q}[c_1,...,c_r]$ of weighted degree $n$ defined as the following determinant:
\begin{align*}
  s_\lambda (c_1,...,c_r)&=\det [c_{\lambda_j -j +k}]_{1\leq j, k \leq n}\\
  &= \left|
                                                                             \begin{array}{cccc}
                                                                               c_{\lambda_1} & c_{\lambda_1 +1} & ... & c_{\lambda_1 +n-1} \\
                                                                               c_{\lambda_2 -1} & c_{\lambda_2} & ... & c_{\lambda_2 +n-2} \\
                                                                               ... & ... & ... & ... \\
                                                                               c_{\lambda_n - n+1} & c_{\lambda_n -n+2} & ... & c_{\lambda_n} \\
                                                                             \end{array}
                                                                           \right|,
\end{align*}
where we use the convention $c_0 =1, c_i =0$ whenever $i\notin [0, r]$. These polynomials form a basis for the vector space of all homogeneous polynomials of weighted degree $n$ in $r$ variables, where the variable $c_i$ is of weighted degree $i$. Geometrically, Schur polynomials appear in describing the cohomology of Grassmannians (see e.g. \cite[Chapter 14]{fultonIntersectionBOOK2nd}).

\begin{exmple}
We list some simple Schur polynomials.
\begin{enumerate}
  \item For small $n$, the Schur polynomials with $\lambda \in \Lambda(n, n)$, are given by
\begin{align*}
  &n=1: s_{(1)} = c_1,\\
  &n=2: s_{(2, 0)}= c_2,\ s_{(1, 1)} = c_1 ^2- c_2,\\
  &n=3: s_{(3,0,0)} = c_3,\ s_{(2,1,0)}  = c_1 c_2 -c_3,s_{(1,1,1)}  = c_1 ^3 -2 c_1 c_2 +c_3.
\end{align*}
  \item For general $n$, we have
  \begin{equation*}
s_\lambda(c_1, ..., c_r)=
\left\{
  \begin{array}{ll}
    c_{n}, &\ \text{when}\ \lambda = (n, 0,...,0); \\
    c_j  c_{n-j}  -c_{j-1} c_{n+1-j},  &\ \text{when}\ \lambda = (n-j, j,0,...,0).\\
  \end{array}
\right.
\end{equation*}
  \item For $\lambda = (1, 1,...,1,0,...,0)$, $s_\lambda (c_1, ..., c_r)$ is the Segre polynomial.
\end{enumerate}

\end{exmple}

\begin{defn}
Let $X$ be a smooth projective variety of dimension $n$, and $E$ a holomorphic vector bundle of rank $r$ on $X$. Let $c_1 (E),...,c_r (E)$ be the Chern classes of $E$, then the Schur classes of $E$ are defined by
\begin{equation*}
  s_\lambda (E) = s_\lambda (c_1 (E), ..., c_r (E)) \in H^{|\lambda|, |\lambda|} (X, \mathbb{R}).
\end{equation*}

\end{defn}

The definition of Chern classes and Schur classes also make sense for $\mathbb{R}$-twisted bundles (see \cite[Chapter 8]{lazarsfeldPosII}).

In particular, let $\lambda$ be a partition, for each $0\leq i \leq |\lambda|$ the ``derived'' Schur classes (see \cite{ross2019hodge}) $s_\lambda ^{(i)} (E)$ are defined by requiring that
\begin{equation}\label{eq schur}
  s_\lambda (E\langle \delta \rangle)=\sum_{i=0} ^{|\lambda|}  s_\lambda ^{(i)} (E)\cdot \delta^i.
\end{equation}
Then $s_\lambda ^{(i)} (E) \in H^{|\lambda|-i, |\lambda|-i} (X, \mathbb{R})$, with $s_\lambda ^{(0)} (E) =s_\lambda (E)$.
(\ref{eq schur}) can be understood as follows. Let $x_1,..., x_r$ be the Chern roots of $E$, then $x_1 + \delta, ..., x_r +\delta$ are the Chern roots of $E\langle \delta \rangle$. The Schur class is a polynomial of Chern roots, thus one can write
\begin{equation*}
  s_\lambda (x_1 + \delta,...,x_r + \delta) = \sum_{i=0} ^{|\lambda|} s_\lambda ^{(i)} (x_1, ..., x_r) \delta^i,
\end{equation*}
where $s_\lambda ^{(i)} (x_1, ..., x_r)$ is a symmetric polynomial of degree $|\lambda|-i$, giving the class $s_\lambda ^{(i)} (E)$.
This is motivated by the identity
\begin{equation*}
s_\lambda (E\langle \delta \rangle) = s_\lambda (E \otimes L),
\end{equation*}
when $\delta = c_1 (L)$ for some line bundle $L$.

One also has:
\begin{equation*}
  s_\lambda ^{(i)} (E\langle \delta \rangle) =\sum_{k=i}^{|\lambda|} \left(
                                                                       \begin{array}{c}
                                                                         k \\
                                                                         i \\
                                                                       \end{array}
                                                                     \right)
   s_\lambda ^{(k)} (E) \cdot \delta^{k-i}.
\end{equation*}

\begin{rmk}
The classes $s_\lambda (E), s_\lambda ^{(i)} (E)$ and their twisted variants have functorial properties under pullbacks, just as Chern classes.
\end{rmk}

\subsection{Positivity of smooth forms}

We recall several positivity notions on $(p, p)$ forms. The standard reference is \cite{Dem_AGbook}.
Let $(z_1,...,z_n)$ be the Euclidean coordinates on $\mathbb{C}^n$, and $\Lambda^{p, q} (\mathbb{C}^n)$ the space of $(p, q)$ forms on $\mathbb{C}^n$ with constant coefficients. Denote the volume form of $\mathbb{C}^n$ by
\begin{equation*}
d\vol_{\mathbb{C}^n} = idz_1 \wedge d\bar{z}_1 \wedge...\wedge idz_n \wedge d\bar{z}_n.
\end{equation*}

\begin{defn}

A $(p, p)$ form $u \in \Lambda^{p, p} (\mathbb{C}^n)$ is said to be positive if for any $\alpha_j \in \Lambda^{1, 0} (\mathbb{C}^n)$, $1\leq j \leq n-p$, the $(n, n)$-form
$$u \wedge i \alpha_1 \wedge \overline{\alpha_1} \wedge...\wedge i \alpha_{n-p} \wedge \overline{\alpha_{n-p}}$$
differs with $d\vol_{\mathbb{C}^n}$ by a non-negative multiplier.
A $(p, p)$ form $u \in \Lambda^{p, p} (\mathbb{C}^n)$ is said to be strongly positive if $u$ is a non-negative combination
\begin{equation*}
  u = \sum_{s=1} ^m a_s i \alpha_{s, 1} \wedge \overline{\alpha_{s, 1}} \wedge...\wedge i \alpha_{s, p} \wedge \overline{\alpha_{s, p}},
\end{equation*}
where $\alpha_{s, j} \in \Lambda^{1, 0} (\mathbb{C}^n)$ and $a_s \geq 0$.

\end{defn}

The set of positive $(p,p)$ forms is a closed convex cone in $\Lambda^{p, p} (\mathbb{C}^n)$, similarly for the set of strongly positive $(n-p, n-p)$ forms. These two cones are dual to each other via the pairing between $\Lambda^{p, p} (\mathbb{C}^n)$ and $\Lambda^{n-p, n-p} (\mathbb{C}^n)$.
It is easy to see that strongly positive forms must be positive.

\begin{rmk}
For forms of bidegrees $(0,0), (1,1), (n-1, n-1), (n,n)$, the above two positivity notions are equivalent. Furthermore, an $(1,1)$ form
$$u=i\sum_{j,k} u_{jk} dz_j \wedge d\bar{z}_k$$
is positive or strongly positive if and only if the Hermitian matrix $[u_{jk}]_{1\leq j, k\leq n}$ is semi-positive. Denote
$\widehat{dz_j \wedge d\bar{z}_k}$ to be the $(n-1, n-1)$-form such that
$$i dz_j \wedge d\bar{z}_k \wedge \widehat{dz_j \wedge d\bar{z}_k} = d\vol_{\mathbb{C}^n} .$$
Then an $(n-1,n-1)$ form
$$u=\sum_{j,k} u_{jk} \widehat{dz_j \wedge d\bar{z}_k}$$
is positive or strongly positive if and only if the Hermitian matrix $[u_{jk}]_{1\leq j, k\leq n}$ is semi-positive.
\end{rmk}

\begin{defn}
For an $(1,1)$ or $(n-1, n-1)$ form $u$, we call it strictly positive if the Hermitian matrix $[u_{jk}]_{1\leq j, k\leq n}$ is positive definite.

\end{defn}

It is clear that these positivity notions can be also formulated on a complex manifold $X$ by requiring the corresponding positivity at every point of $X$.
Using the duality between forms and currents, one can also define positivity for currents (see \cite[Chapter 3]{Dem_AGbook}). For example, on a compact complex manifold a $(k,k)$ current is positive (resp. strongly positive) if it takes non-negative values on all smooth strongly positive (resp. positive) $(n-k, n-k)$ forms.

\section{Proof of the main result} \label{sec proof}
In this section, we give the proof of the main result.
As mentioned in the introduction, the main result holds in the $\mathbb{R}$-twisted setting, we are going to prove:

\begin{thrm}
Let $X$ be a smooth projective variety of dimension $n$, and let $E$ be a $\mathbb{R}$-twisted ample vector bundle of rank $r$ on $X$. Then for any Schur polynomial $s_\lambda (c_1,...,c_r)$ with $|\lambda|= n-1$, whenever it is non-zero the cohomology class $s_\lambda(E)$ has a smooth representative which is strictly positive in the sense of $(n-1, n-1)$ forms.
\end{thrm}

\begin{proof}
Since on a surface the Schur class of bidegree $(1,1)$ is just the first Chern class, without loss of generality, in the following we always assume that $n\geq 3$.

We need to prove the existence of a strictly positive representative in an $(n-1, n-1)$ class. The idea is to apply the following very useful positivity criterion by using duality between the pseudo-effective and the movable cone \cite[Corollary A]{nystromDualityMorse} (see also \cite[Appendix]{fuxiao14kcone}, \cite{toma2010note}, \cite{BDPP13}).

\begin{lem}\label{lem pos criterion}
Let $X$ be a smooth projective variety of dimension $n$ and $\alpha \in H^{n-1, n-1}(X, \mathbb{R})$, then $\alpha$ has a smooth representative which is strictly positive if and only if for any non-zero pseudo-effective $(1,1)$ class $L \in H^{1,1} (X, \mathbb{R})$, the intersection number $\alpha \cdot L >0$. Moreover, for $\alpha \in H^{n-1, n-1}(X, \mathbb{Q})$, one only needs to check the positivity condition $\alpha \cdot L >0$ for $L$ rational.
\end{lem}

Recall that $L \in H^{1,1} (X, \mathbb{R})$ is called pseudo-effective if it has a representative which is positive in the sense of $(1,1)$ currents.

Applying Lemma \ref{lem pos criterion} in our setting, it is sufficient to check that:
\begin{equation}\label{eq pos}
s_\lambda (E) \cdot L >0
\end{equation}
for any non-zero pseudo-effective class $L \in H^{1,1} (X, \mathbb{R})$. If $L$ is the class of a hypersurface, then it follows from \cite{fultonlazarsfeldPositive}; if $L$ is a K\"ahler class, then the non-negativity follows from \cite[Theorem 2.5]{DPSneftangent}.

We reduce the proof to the above two cases. To this end, we apply the divisorial Zariski decomposition due to \cite{Bou04} (see also \cite{Nak04}):

\begin{lem}\label{lem zar}
Let $X$ be a smooth projective variety (or more generally, a compact K\"ahler manifold) of dimension $n$, let $L \in H^{1,1} (X, \mathbb{R})$ be a pseudo-effective class. Then there exists a decomposition
\begin{equation*}
  L = P(L) + N(L),
\end{equation*}
such that $P(L)$ is movable and $N(L)$ is effective.
\end{lem}

More precisely, $ N(L)= \sum_{i=1} ^ m a_i [D_i]$ for some $a_i \geq 0$ and some prime divisors $D_i$, and $P(L)$ can be written as
$$P(L) = \lim_m (\pi_m)_* \widehat{\omega}_m,$$
where $\pi_m : X_m \rightarrow X$ is a modification and $\widehat{\omega}_m$ is a K\"ahler class on the projective manifold $X_m$.

Now we fix a non-zero pseudo-effective class $L \in H^{1,1} (X, \mathbb{R})$, and apply the divisorial Zariski decomposition to $L$.

If $N(L) = \sum a_j [D_j] \neq 0$, i.e., in the summands there exists a prime divisor $D_{j_0}$ such that $a_{j_0}>0$, then by \cite{fultonlazarsfeldPositive},
\begin{equation*}
s_\lambda (E)\cdot [D_{j_0}] = s_\lambda (E|_{D_{j_0}}) >0.
\end{equation*}

If $N(L)=0$, i.e., $L =P(L)$ is a non-zero movable class, then there is a sequence of modifications $\pi_m : X_m \rightarrow X$ and K\"ahler classes $\widehat{\omega}_m$ on $X_m$ such that
\begin{equation*}
  L = \lim_m  (\pi_m)_* \widehat{\omega}_m.
\end{equation*}
Thus,
\begin{align*}
s_\lambda (E) \cdot L &= \lim_m s_\lambda (E) \cdot (\pi_m)_* \widehat{\omega}_m\\
 &= \lim_m s_\lambda (\pi_m ^* E) \cdot \widehat{\omega}_m \geq 0
\end{align*}
where the last inequality holds by \cite[Theorem 2.5]{DPSneftangent}, since $\pi_m ^* E $ is nef on $X_m$.

In particular, we have proved:

\begin{lem}\label{lem non-negative mov}
For $E$ a $\mathbb{R}$-twisted nef vector bundle, $s_\lambda (E) \cdot L \geq 0$ for any non-zero movable class $L$.
\end{lem}

Combining the above discussions, we only need to prove (\ref{eq pos}) when $L$ is a non-zero movable class.

In the following, we assume that $L$ is a non-zero movable class.

Similar to \cite{blochgieseker} (see also \cite[Chapter 8]{lazarsfeldPosII}), we use the ampleness of $E$ and a perturbation argument to prove that $s_\lambda (E) \cdot L >0$.

Fix a very ample line bundle $H$ such that $\omega =c_1 (H)$ on $X$. Since $E$ is ample, for $c >0$ sufficiently small and $t\in [0, c]$, the $\mathbb{R}$-twisted bundle $E\langle -t \omega\rangle$ is also ample, thus by Lemma \ref{lem non-negative mov}
\begin{equation*}
  s_\lambda (E\langle -t \omega\rangle) \cdot L \geq 0, \ \text{for any}\  t\in [0, c].
\end{equation*}

Recall that the Schur classes $s_\lambda (E\langle -t \omega\rangle)$ and $s_\lambda (E)$ are related by the following identity (see (\ref{eq schur})):
\begin{equation*}
  s_\lambda (E\langle -t \omega\rangle) = s_\lambda (E) - t s_\lambda ^{(1)} (E) \cdot \omega + O(t^2),
\end{equation*}
where $s_\lambda ^{(1)} (E) \in H^{n-2, n-2} (X, \mathbb{R})$, and $O(t^2)$ is the term of order $t^2$. Therefore,
\begin{equation*}
  s_\lambda (E) \cdot L =  s_\lambda (E\langle -t \omega\rangle) \cdot L + t s_\lambda ^{(1)} (E) \cdot \omega \cdot L + O(t^2).
\end{equation*}

To finish the proof, we only need to check that $s_\lambda ^{(1)} (E) \cdot \omega \cdot L >0$.

The idea is to apply the Hodge index theorem due to Ross-Toma \cite{ross2019hodge}. Roughly speaking, we first verify that
\begin{equation}\label{eq non-negative}
s_\lambda ^{(1)} (E) \cdot \omega \cdot L \geq0 \ \text{and}\ s_\lambda ^{(1)} (E) \cdot L \cdot L \geq0.
\end{equation}
Assuming that $s_\lambda ^{(1)} (E) \cdot \omega \cdot L =0$, i.e., $L$ is primitive with respect to $(s_\lambda ^{(1)} (E), \omega)$, then by the Hodge index theorem,
\begin{equation*}
s_\lambda ^{(1)} (E) \cdot L \cdot L \leq0
\end{equation*}
with equality holds only if $L=0$. Combining with $s_\lambda ^{(1)} (E) \cdot L \cdot L \geq0$ we get $L=0$, which contradicts with our assumption $L \neq 0$.

\begin{lem}\label{lem hodge index}
The class $s_\lambda ^{(1)} (E) \in H^{n-2, n-2} (X, \mathbb{R})$ satisfies Hodge index theorem, and (\ref{eq non-negative}) holds.
\end{lem}

The proof of Lemma \ref{lem hodge index} mainly follows from \cite{ross2019hodge}.

Fix a complex vector space $V$ of dimension $\dim V = r+n$, and let $\underline{V} = X \times V$ be the trivial vector bundle on $X$. Let $F = \underline{V} \otimes E$, and denote the rank of $F$ by $f+1 = (r+n) r$.

Let $p: P(F)\rightarrow X$ be the projective bundle of lines of $F$, and $U$ the universal quotient bundle on $P(F)$:
\begin{equation*}
0\rightarrow \mathcal{O}_{P(F)} (-1) \rightarrow p^* F \rightarrow U \rightarrow 0.
\end{equation*}
The bundle $U$ is of rank $f$, and it is nef when $F$ is nef.
By \cite[Proposition 5.2]{ross2019hodge} (though it is stated for certain degrees in their setting, it also holds for general degrees),
\begin{equation*}
  s_\lambda ^{(1)} (E)\cdot \beta = p_* c_{f-1}(U_{|C}) \cdot \beta
\end{equation*}
for any $\beta \in H^{1,1}(X, \mathbb{R})$, where $C \subset P(F)$ is an irreducible subvariety of dimension $f+1$ and locally a product over $X$ (see \cite{ross2019hodge} or \cite{fultonIntersectionBOOK2nd}). Therefore, by the projection formula
\begin{equation}\label{eq proj}
  s_\lambda ^{(1)} (E)\cdot \beta \cdot \beta' =  c_{f-1}(U_{|C}) \cdot p^* \beta \cdot p^*\beta',
\end{equation}
where $\beta, \beta' \in H^{1,1}(X, \mathbb{R})$.

The bundle $E$ being ample implies that $F$ is ample. By \cite[Theorem 4.1]{ross2019hodge}, the quadratic form
\begin{align*}
  Q(\beta, \beta') :&= c_{f-1}(U_{|C}) \cdot p^* \beta \cdot p^*\beta'\\
  &=s_\lambda ^{(1)} (E)\cdot \beta \cdot \beta'
\end{align*}
satisfies Hodge index theorem, that is, it has signature $(1, h^{1,1}(X) -1)$.

As for (\ref{eq non-negative}), we first claim that $s_\lambda ^{(1)} (E)\cdot \omega \cdot L \geq 0$. Recall that $L=\lim  (\pi_m)_* \widehat{\omega}_m$, thus
\begin{equation*}
  s_\lambda ^{(1)} (E)\cdot \omega \cdot L = \lim_m s_\lambda ^{(1)} (\pi_m ^* E)\cdot \pi_m ^*\omega \cdot \widehat{\omega}_m.
\end{equation*}
Applying the above argument and (\ref{eq proj}) to the nef bundle $\pi_m ^* E$ on $X_m$ shows that every term
\begin{align*}
&s_\lambda ^{(1)} (\pi_m ^* E)\cdot \pi_m ^*\omega \cdot \widehat{\omega}_m\\
&= c_{f-1}({U_m}_{|C_{m}}) \cdot p_m^* (\pi_m ^*\omega) \cdot p_m ^*\widehat{\omega}_m\\
& \geq 0,
\end{align*}
where the second inequality follows from \cite[Corollary 2.2]{DPSneftangent} and $\omega$ being an ample divisor class. Here we denote the corresponding objects by the subscript $m$.
Thus the limit is also non-negative, finishing the proof of our claim.

Next we claim that $s_\lambda ^{(1)} (E)\cdot L \cdot L \geq 0$. The proof is similar to the first claim. Note that
\begin{equation*}
  s_\lambda ^{(1)} (E)\cdot L \cdot L = \lim_m s_\lambda ^{(1)} (\pi_m ^* E)\cdot \pi_m ^* (\pi_{m*}\widehat{\omega}_m) \cdot \widehat{\omega}_m.
\end{equation*}
By Siu decomposition \cite{siuDecomposition}, we have\footnote{This is the transcendental analogy of Negativity Lemma in algebraic geometry.}
\begin{equation*}
\pi_m ^* (\pi_{m*}\widehat{\omega}_m) = \widehat{\omega}_m + [D]
\end{equation*}
for some effective $\mathbb{R}$-divisor $D$. Therefore,
\begin{align*}
 & s_\lambda ^{(1)} (\pi_m ^* E)\cdot \pi_m ^* (\pi_{m*} \widehat{\omega}_m) \cdot \widehat{\omega}_m\\
 & =
  s_\lambda ^{(1)} (\pi_m ^* E)\cdot [D] \cdot \widehat{\omega}_m + s_\lambda ^{(1)} (\pi_m ^* E)\cdot \widehat{\omega}_m \cdot \widehat{\omega}_m.
\end{align*}
Repeating the argument for $\pi_m ^* E$ on $X_m$ as the first claim and applying \cite{DPSneftangent} show that every summand is non-negative. This finishes the proof of the second claim.

By using \cite[Theorem 4.3]{ross2019hodge} or \cite{fultonlazarsfeldPositive},
\begin{equation*}
Q(\omega, \omega) =c_{f-1}(U_{|C}) \cdot p^* \omega^2 >0.
\end{equation*}
Therefore, if $Q(\omega, L)=s_\lambda ^{(1)} (E) \cdot \omega \cdot L =0$, then by Lemma \ref{lem hodge index} $Q(L, L) \leq 0$ with the equality holds if and only if $L=0$. By Lemma \ref{lem hodge index} again, $Q(L, L) \geq 0$, thus
$s_\lambda ^{(1)} (E) \cdot \omega \cdot L =0$ yields $L=0$,
contradicting with $L\neq 0$.

Therefore, $s_\lambda ^{(1)} (E) \cdot \omega \cdot L >0$.
As $L$ is arbitrary, this finishes the proof of the theorem.

\end{proof}

\begin{rmk}
When $s_\lambda (E)= c_{n-1} (E)$ is the $(n-1)$-th Chern class of $E$ and the rank $r\geq n-1$, the estimate $c_{n-1} (E)\cdot L >0$ can be alternatively proved by an inductive argument.

Recall that in our perturbation argument, we assume that $\omega = c_{1} (H)$ for a very ample line bundle $H$. Use the same notation $H$ to denote a very general smooth hypersurface in the linear system of $H$. Then by the identity
\begin{equation}\label{eq chern}
  c_k (E\langle \delta \rangle) = \sum_{i=0} ^k \left(
                                                  \begin{array}{c}
                                                    r-i \\
                                                    k-i \\
                                                  \end{array}
                                                \right)
                                                c_i (E) \delta^{k-i}, \ 0\leq k \leq r,
\end{equation}
we have
\begin{align*}
  c_{n-1} (E) \cdot L &=  c_{n-1} (E\langle -t \omega\rangle) \cdot L + t (r-n+2)c_{n-2} (E) \cdot \omega \cdot L + O(t^2)\\
  &=c_{n-1} (E\langle -t \omega\rangle) \cdot L + t (r-n+2)c_{n-2} (E_{|H}) \cdot L_{|H} + O(t^2).
\end{align*}
By the weak Lefschetz theorem, the restriction $i^*: H^{2} (X, \mathbb{R})\rightarrow H^{2} (H, \mathbb{R})$ is injective whenever $2\leq n-1$. Our assumption $n\geq 3$ shows that the weak Lefschetz theorem holds.
Thus the class $L_{|H}$ is non-zero. As $H$ is very general, it has a representative which is a non-zero positive $(1,1)$ current on $H$.

By induction, $c_{n-2} (E_{|H})$ has a representative which is a smooth strictly positive $(n-2, n-2)$ form on the hypersurface $H$.
This yields that
\begin{equation*}
c_{n-2} (E) \cdot \omega \cdot L = c_{n-2} (E_{|H}) \cdot L_{|H}>0,
\end{equation*}
finishing the proof for $c_{n-1} (E)$.

As for an inductive argument to the general Schur class $s_\lambda (E)$, it is not clear to us if $s_\lambda ^{(1)} (E)$ is numerically positive, or equivalently, a non-negative linear combination of Schur polynomials of weighted degree $(n-2)$.
\end{rmk}

Corollary \ref{intro cor1} follows from applying Theorem \ref{intro main thrm} to the subvariety $V$ and the weak Lefschetz theorem.

\begin{proof}[Proof of Corollary \ref{intro cor1}]
We can assume that $n\geq 3$. By the proof of Theorem \ref{intro main thrm}, we only need to check the positivity of intersection numbers with non-zero pseudo-effective $(1,1)$ classes.

Since $V$ is the complete intersection of very ample divisors and we only care about intersection numbers, we can assume that $V$ is the intersection of very general elements of the linear systems of the $H_i$. %Given a prime divisor $D$,
%by \cite{fultonlazarsfeldPositive} and the weak Lefschetz theorem we have
%\begin{equation*}
%s_\lambda (E) \cdot [V] \cdot D = s_\lambda (E_{|V}) \cdot  D_{|V}>0.
%\end{equation*}
Given a non-zero pseudo-effective $(1,1)$ class $L$, by weak Lefschetz theorem its restriction $L_{|V}$ is non-zero. As $V$ is very general, it is also pseudo-effective by \cite{DemaillyRegularization}. Thus applying Theorem \ref{intro main thrm} on $V$ shows that
\begin{equation*}
s_\lambda (E) \cdot [V] \cdot L = s_\lambda (E_{|V}) \cdot  L_{|V}>0.
\end{equation*}

This finishes the proof of Corollary \ref{intro cor1}.
\end{proof}

\begin{rmk}
For a nef vector bundle $E$ over a compact K\"ahler manifold $X$ of dimension $n$, \cite{DPSneftangent} and the proof of Theorem \ref{intro main thrm} imply that the Schur class $s_\lambda (E)$ with $|\lambda|=n-1$ is always in the dual cone of the pseudo-effective cone of $(1,1)$ classes.
\end{rmk}

\subsection*{Acknowledgements}
This work was supported in part by Tsinghua University Initiative Scientific Research Program (No. 2019Z07L02016) and NSFC (No. 11901336).

\bibliography{reference}
\bibliographystyle{amsalpha}

\bigskip

\bigskip

\noindent
\textsc{Department of Mathematical Sciences and Yau Mathematical Sciences Center}\\
\textsc{Tsinghua University, Beijing 100084, China}\\
\noindent
\verb"Email: jianxiao@tsinghua.edu.cn"

\end{document}